\documentclass[leqno,11pt]{amsart}

\usepackage{bm}
\usepackage{amssymb,amsfonts}
\usepackage{amsmath,latexsym}
\usepackage{pdfsync}
\usepackage{xcolor}
\usepackage{mathrsfs}

\newtheorem{thm}{Theorem}[section]

\newtheorem{lem}[thm]{Lemma}

\theoremstyle{definition}
\newtheorem{defn}[thm]{Definition}
\newtheorem{rem}[thm]{Remark}
\numberwithin{equation}{section}
\def\eqdefa{\buildrel\hbox{\footnotesize def}\over =}
\newcommand{\R}{\mathbb R}
\newcommand\E{\mathcal{E}}
\newcommand\J{\mathscr J}
\newcommand\B{\mathscr{B}}
\newcommand\N{\mathbb{N}}
\newcommand{\dy}{\,{\rm d}y}
\newcommand{\dx}{\,{\rm d}x}
\newcommand{\dt}{\,{\rm d}t}
\newcommand{\ds}{\,{\rm d}s}

\def\Xint#1{\mathchoice	
	{\XXint\displaystyle\textstyle{#1}}	{\XXint\textstyle\scriptstyle{#1}}	{\XXint\scriptstyle\scriptscriptstyle{#1}}	{\XXint\scriptscriptstyle\scriptscriptstyle{#1}}	\!\int}
\def\XXint#1#2#3{{\setbox0=\hbox{$#1{#2#3}{\int}$}	\vcenter{\hbox{$#2#3$}}\kern-.5\wd0}}

\def\dashint{\Xint-}
\allowdisplaybreaks[4]
\begin{document}
\thispagestyle{empty}

\vspace{1 true cm} {
\title[Regularity for the solutions of 3D MHD equations]
 {Serrin--type regularity criteria for the 3D MHD equations via one  velocity component and one magnetic component}%
\author[H.Chen]{Hui Chen}%
\address[H. Chen]
 {School of Science, Zhejiang University of Science and Technology, Hangzhou, 310023, People's Republic of China }
\email{chenhui@zust.edu.cn}
\author[C. Qian]{Chenyin Qian}
\address[C. Qian]%
{Department of Mathematics, Zhejiang  Normal University Jinhua, 321004, China}
\email{qcyjcsx@163.com }
\author[T. Zhang]{Ting Zhang*}

\address[T. Zhang]{School of Mathematical Sciences, Zhejiang University,  Hangzhou 310027, People's Republic of China}

\email{zhangting79@zju.edu.cn}
\thanks{$^*$Corresponding author.}

\begin{abstract}
In this paper, we consider the Cauchy problem to the 3D MHD equations. We show that the Serrin--type conditions imposed on one component of the velocity $u_{3}$ and one component of magnetic fields $b_{3}$  with
$$
u_{3} \in L^{p_{0},1}(-1,0;L^{q_{0}}(B(2))),\ b_{3} \in L^{p_{1},1}(-1,0;L^{q_{1}}(B(2))),
$$
$\frac{2}{p_{0}}+\frac{3}{q_{0}}=\frac{2}{p_{1}}+\frac{3}{q_{1}}=1$ and $3<q_{0},q_{1}<+\infty$ imply that the suitable weak solution is regular at $(0,0)$. The proof is based on the new local energy estimates introduced by Chae-Wolf (Arch. Ration. Mech. Anal. 2021) and Wang-Wu-Zhang (arXiv:2005.11906).
\end{abstract}

\maketitle



\vskip 0.2cm


\section{Introduction}
\label{intro}
We consider the Cauchy problem to the three-dimensional incompressible magnetohydrodynamics
(MHD) equations in $\R^3\times(0,\infty)$
\begin{equation}\label{MHD}
	\left\{
	\begin{array}{ll}
		\vspace{4pt}
		&\partial_{t}\bm{u}+(\bm{u}\cdot\nabla) \bm{u}-\mu\Delta  \bm{u}+\nabla \pi=\bm{b}\cdot\nabla \bm{b},\\
		\vspace{4pt}
		&\partial_{t}\bm{b}+(\bm{u}\cdot\nabla)\bm{b}-\nu\Delta\bm{b}=\bm{b}\cdot\nabla\bm{u},\\
		&\nabla\cdot  \bm{u}=\nabla\cdot\bm{b}=0~,
	\end{array}
	\right.
\end{equation}
with initial data
\begin{align}\label{MHDi}
	\bm{u}|_{t=0}=\bm{u}_{0},~\bm{b}|_{t=0}=\bm{b}_{0}.
\end{align}
Here  $\bm{u}=(u_{1},u_{2},u_{3}),\ \bm{b}=(b_{1},b_{2},b_{3})$ and $ \pi$ are nondimensional quantities corresponding to the
velocity, magnetic fields and a scalar pressure. $\mu>0$ is the viscosity coefficient and $\nu>0$ is the resistivity coefficient.  We require $\mu=\nu=1$ in this paper.

Such system \eqref{MHD}--\eqref{MHDi} describes many phenomena such as the geomagnetic dynamo in geophysics, solar winds and solar flares in astrophysics. G. Duvaut and J. L. Lions \cite{Duvaut1972}  constructed a global weak solution and the local strong solution to the initial boundary value problem, and the properties of such
solutions have been examined by M. Sermange and R. Temam in  \cite{Sermange1983}.

If the magnetic fields $\bm{b}=0$, the system \eqref{MHD} degenerates to the incompressible Navier--Stokes equations. A global weak solution to the Navier--Stokes equations was constructed by J. Leray \cite{Leray1934}.  However, the uniqueness and regularity of such weak solution is still one of the most challenging open problems in the field of mathematical fluid mechanics. One essential work is usually referred as Serrin--type conditions (see \cite{Escauriaza2003,Prodi1959,Serrin1962,Takahashi1990} and the references therein.), i.e. if the weak solution $\bm{u}$  satisfies
\begin{equation} \bm u\in L^{p}(0,T; L^{q}(\R^{3})),\ \
	\frac{2}{p}+\frac{3}{q}=1,\ 3\leqslant q\leqslant \infty,
\end{equation} then the weak solution is regular in $(0,T]$. 
There are several notable results \cite{Chae2021,Chemin2016,Chemin2017,Chen2020,Han2019,Wang2020} to weaken the above criteria by imposing constraints only on partial components or directional derivatives of velocity field. In particular, D. Chae and J. Wolf \cite{Chae2021} made an important progress and obtained the regularity of solution  under the condition
\begin{align}\label{1.8}
	u_{3} \in L^{p}\left(0, T ; L^{q}\left(\R^{3}\right)\right), \quad \frac{2}{p}+\frac{3}{q}<1, \quad 3<q\leqslant\infty.
\end{align}
W. Wang, D. Wu and Z. Zhang \cite{Wang2020} improved to
\begin{align}\label{1.8a}
	u_{3} \in L^{p,1}\left(0, T; L^{q}\left(\R^{3}\right)\right), \quad \frac{2}{p}+\frac{3}{q}=1, \quad 3<q<\infty.
\end{align}
Throughout this paper, $L^{p,1}$ denotes the Lorentz space with respect to the time variable.

If the magnetic fields $\bm{b}\neq0$, the situation is more complicated due to the coupling effect between the velocity $\bm{u}$ and the magnetic fields $\bm{b}$. Some fundamental Serrin--type regularity criteria in term of the velocity only were done in \cite{Chen2008,He2005a,Wang2007,Zhou2005}. For instance, if the weak solution $\left(\bm{u},\bm{b}\right)$  satisfies
\begin{equation} \bm u\in L^{p}(0,T; L^{q}(\R^{3})),\ \
	\frac{2}{p}+\frac{3}{q}=1,\ 3< q\leqslant \infty,
\end{equation} then the solution is regular in $(0,T]$.
More regularity criteria can be found in \cite{Han2020,Zhang2015} and the references therein.

We now recall the notion of suitable weak solution to the MHD equations \eqref{MHD}.
\begin{defn}
	We say that  $(\bm{u},\bm{b}, \pi)$ is a suitable weak solution to the MHD equations \eqref{MHD} in an open domain $\Omega_{T}=\Omega \times(-T, 0)$, if
	\begin{enumerate}
		\item[$(1)$] $\bm{u},\bm{b} \in L^{\infty}\left(-T, 0 ; L^{2}(\Omega)\right) \cap L^{2}\left(-T, 0 ; H^{1}(\Omega)\right)$ and $\pi \in L^{\frac{3}{2}}\left(\Omega_{T}\right)$;
		\item[$(2)$] \eqref{MHD} is satisfied in the sense of distribution;
		\item[$(3)$] the local energy inequality holds: for any nonnegative test function\\ $\varphi \in C_{c}^{\infty}\left(\Omega_{T}\right)$  and  $ t \in (-T, 0)$,
		\begin{align}\label{local energy}
			&\int_{\Omega}\left(|\bm{u}|^{2}+|\bm{b}|^{2}\right) \varphi \dx+2 \int_{-T}^{t} \int_{\Omega}\left(|\nabla \bm{u}|^{2}+|\nabla \bm{b}|^{2}\right) \varphi \dx\ds \notag\\
			\leqslant& \int_{-T}^{t} \int_{\Omega}\left(|\bm{u}|^{2}+|\bm{b}|^{2}\right)\cdot\left(\partial_{s} \varphi+\Delta \varphi\right)+\bm{u} \cdot \nabla \varphi\left(|\bm{u}|^{2}+|\bm{b}|^2+2 \pi\right) \dx\ds\notag\\
			&-2\int_{-T}^{t} \int_{\Omega}\left(\bm{u}\cdot\bm{b}\right)\bm{b}\cdot\nabla\varphi\dx\ds.
		\end{align}
	\end{enumerate}
\end{defn}
The global existence of suitable weak solution to the MHD equations \eqref{MHD} was investigated by C. He and Z. Xin \cite{He2005}. They also obtained that one-dimensional Hausdorff measure of the possible space-time singular points set for the suitable weak solution is zero.

In this paper, inspired by \cite{Chae2021,Wang2020}, we focus on the regularity criteria to the MHD equation \eqref{MHD} involving the terms $u_{3}$ and $b_{3}$ only.
\begin{thm}\label{thm1}
	Let $(\bm{u},\bm{b}, \pi)$ be a suitable weak solution of the MHD equations \eqref{MHD} in $\R^3 \times(-1,0) .$ If $\bm{u},\bm{b}$ satisfies
	\begin{equation}\label{R1}
		u_{3} \in L^{p_{0},1}\left(-1,0;L^{q_{0}}(B(2))\right)
	\end{equation}
	and
	\begin{equation}\label{R2} b_{3} \in L^{p_{1},1}\left(-1,0;L^{q_{1}}(B(2))\right)
	\end{equation}
	with $\frac{2}{p_{0}}+\frac{3}{q_{0}}=\frac{2}{p_{1}}+\frac{3}{q_{1}}=1,~3 < q_{0},q_{1}<+ \infty$, then the solution $(\bm{u},\bm{b})$ is regular at $(0,0)$.  Here $B(r)$ is the ball in $\R^3$ with center at origin and radius $r$.
\end{thm}
\begin{rem}
	It should be pointed out that the condition  $\mu=\nu$ in the MHD equations \eqref{MHD} is essential in our estimates. Actually, we adopt the Els\"{a}sser variables  $\bm{Z}^{+}=\bm{u}+\bm{b}$ and $\bm{Z}^{-}=\bm{u}-\bm{b}$ to obtain the crucial lemma \ref{A8} that if $u_{3}=b_{3}=0$, the solution is regular at $(0,0)$. However, it is not trivial in the case $\mu\neq\nu$. We will discuss it in our future work.
\end{rem}
\begin{rem}\label{rem1.4}
	If we replace \eqref{R1} with the subcritical regularity criteria
	\begin{equation}
		u_{3} \in L^{p_{0}}\left(-1,0;L^{q_{0}}(B(2))\right),\ \frac{2}{p_{0}}+\frac{3}{q_{0}}<1,\ 3<q_{0}\leqslant +\infty,
	\end{equation}
	Theorem \ref{thm1} still holds.
	Actually, we can pick $2<{p}_{2}<p_{0}$ and $3<q_{2}<q_{0}$ with $\frac{2}{p_{2}}+\frac{3}{q_{2}}=1$. Therefore, we can prove it directly by the embedding inequality
	\begin{equation*}
		\left\|u_{3}\right\|_{L^{p_{2},1}\left(-1,0;L^{q_{2}}(B(2))\right)}\leqslant C\left\|u_{3}\right\|_{L^{p_{0}}\left(-1,0;L^{q_{0}}(B(2))\right)}.
	\end{equation*}
	Analogously, we can replace \eqref{R2} with the regularity criteria
	\begin{equation}
		b_{3} \in L^{p_{1}}\left(-1,0;L^{q_{1}}(B(2))\right),\ \frac{2}{p_{1}}+\frac{3}{q_{1}}<1,\ 3<q_{1}\leqslant +\infty.
	\end{equation}
\end{rem}
\begin{rem}
	By the standard interpolation theory, it is possible to extend the regularity criteria in Theorem \ref{thm1} to weaker one, such as 
	\begin{equation}
		u_{3} \in L^{p_{0},1}\left(-1,0;L^{q_{0},\infty}(B(2))\right),\ b_{3} \in L^{p_{1},1}\left(-1,0;L^{q_{1},\infty}(B(2))\right),
	\end{equation}
	with $\frac{2}{p_{0}}+\frac{3}{q_{0}}=\frac{2}{p_{1}}+\frac{3}{q_{1}}=1,~3 < q_{0},q_{1}<+ \infty$. We leave the details to the interested readers.
\end{rem}
Thanks to  Theorem  \ref{thm1} and  Remark \ref{rem1.4}, we obtain the following theorem.
\begin{thm}\label{thm2}
	Assume that $\bm{u}_{0},\ \bm{b}_{0}\in  H^1(\R^3)$ with $\textrm{div}\ \bm{u}_{0}=\textrm{div}\  \bm{b}_{0}=0$  and $(\bm u,\bm{b})$ is the  weak solution to the MHD equations \eqref{MHD}--\eqref{MHDi}. Then the solution is regular in $\R^3\times(0,T]$, provided that
	\begin{enumerate}
		\item[$\mathrm{(1)}$] $u_{3} \in L^{p_{0},1}(0,T;L^{q_{0}}(\R^3))$ with $\frac{2}{p_{0}}+\frac{3}{q_{0}}=1$, $3<q_{0}<\infty$  , or 
		
		$u_{3} \in L^{p_{0}}(0,T;L^{q_{0}}(\R^3))$ with $\frac{2}{p_{0}}+\frac{3}{q_{0}}<1$, $3<q_{0}\leqslant+\infty$;\\
		\item[$\mathrm{(2)}$] $b_{3} \in L^{p_{1},1}(0,T;L^{q_{1}}(\R^3))$ with $\frac{2}{p_{1}}+\frac{3}{q_{1}}=1$, $3<q_{1}<\infty$  , or 
		
		$b_{3} \in L^{p_{1}}(0,T;L^{q_{1}}(\R^3))$ with $\frac{2}{p_{1}}+\frac{3}{q_{1}}<1$, $3<q_{1}\leqslant+\infty$.
	\end{enumerate} 
\end{thm}

Our paper is organized as follows: we recall some notations and preliminary results in Section \ref{Sec 2}, and establish a key lemma in Section \ref{Sec 3}; finally, we will complete the proof of Theorem \ref{thm1} in Section \ref{Sec 4}.

\section{Notations and Preliminary}\label{Sec 2}
In this preparation section, we recall some usual notations and  preliminary results.

For two comparable quantities, the inequality $X\lesssim Y$ stands for $X\leqslant C Y$ for some positive constant $C$. The dependence of the constant $C$ on other parameters or constants are usually clear from the context, and we will often suppress this dependence.

\par
We use the following standard notations in the literature:
Given matrix $\bm{A}=(A_{ij})\in \R^{m\times n}$, we denote the norm $|\bm{A}|=\left(\Sigma_{i=1}^{m}\Sigma_{j=1}^{n}A_{ij}^2\right)^{\frac{1}{2}}$; for
two vectors $\bm{a}=(a_{1},a_{2},a_{3}),\bm{b}=(b_{1},b_{2},b_{3})$, $\bm{a}\cdot\bm{b}$ and $|\bm{a}|$ are the usual scalar product and norm in $\R^3$ respectively, and $\bm{a}\otimes\bm{b}$ is a matrix with $\left(\bm{a}\otimes\bm{b}\right)_{ij}=a_{i}b_{j}$, $\left(\bm{a},\bm{b}\right)=(a_{1},a_{2},a_{3},b_{1},b_{2},b_{3})$; for $x=(x_{1},x_{2},x_{3})$ $\in \R^{3}$,  $x^{\prime}=(x_{1},x_{2})$ is the horizontal variable;  $B(x_{0},R)$ is the ball in $\R^3$ with center at $x_{0}$ and radius $R$, and $B(R):= B(0,R)$;  $B^{\prime}(x'_{0},R)\subset \R^2$ is the ball in the horizontal plane  with center at $x'_{0}$ and radius $R$, and $B'(R):=B'(0,R)$; $Q(z_{0},R)=B(x_{0},R)\times(t_{0}-R^2,t_{0})$ with $z_{0}=(x_{0},t_{0})$ and $Q(R):=Q(0,R)$; we denote the integral mean value $\left(f(t)\right)_{B(r)}=\dashint_{B(r)}f(y,t)\dy$.

We  also shall use the same  notation as that in  Chae-Wolf \cite{Chae2021}. Set
\begin{equation*}
	U_n(R) \eqdefa B'(R)\times(-r_n, r_n),\ \ Q_n(R) \eqdefa U_n(R)\times(-r_n^2, 0),
\end{equation*}
and
\begin{equation*}
	A_n(R) \eqdefa Q_n(R) \backslash Q_{n+1}(R),
\end{equation*}
where $r_{n}=2^{-n},\ n\in\N$. We consider $\mathrm{\Phi}_n$ given by
$$
\mathrm{\Phi}_n (x,t)\eqdefa \frac{1}{\sqrt{4\pi(-t+r_n^2)}}e^{-\frac{x_3^2}{4(-t+r_n^2)}},\quad (x,t)\in\R^3\times(-\infty, 0),
$$
which satisfies the fundamental solution of the backward heat equation
\begin{equation*}
	\partial_{t}\mathrm{\Phi}_{n}+\partial_{3}^2 \mathrm{\Phi}_{n}=0.
\end{equation*}
There exist absolute  constants $c_1, c_2>0$ such that for  $j=1,\ldots,n-1$, it holds
\begin{equation}\label{Phi_{n}}
	\begin{split}
		& \mathrm{\Phi}_n\leqslant c_2 r_j^{-1},\qquad \qquad \ \ |\partial_3\mathrm{\Phi}_n|\leqslant c_2 r_j^{-2},\qquad\quad\text{in}\quad A_j(R),\\
		c_1 r_n^{-1}\leqslant& \mathrm{\Phi}_n\leqslant c_2 r_n^{-1},\quad c_1 r_n^{-2}\leqslant |\partial_3\mathrm{\Phi}_n|\leqslant c_2 r_n^{-2},\quad\qquad\text{in}\quad Q_n(R).
	\end{split}
\end{equation}
We denote the energy
\begin{align*}
	E_n(R) &\eqdefa \sup_{t\in (-r_n^2, 0)}\int_{U_n(R)}|\left(\bm{u},\bm{b}\right)|^2\dx+\int_{-r^2_n}^{0}\int_{U_n(R)}|\nabla \left(\bm{u},\bm{b}\right)|^2\dx\dt,\\
	\E&\eqdefa  \sup_{t\in (-1, 0)}\int_{\R^3}|\left(\bm{u},\bm{b}\right)|^2\dx+\int_{-1}^{0}\int_{\R^3}|\nabla \left(\bm{u},\bm{b}\right)|^2 \dx \dt.
\end{align*}

The following lemma ensures the energy estimates.
\begin{lem}\cite[Lemma 3.1]{Chae2021}\label{energy}
	Let $ R\geqslant \frac{1}{2}$. For $ 2\leqslant p \leqslant \infty$, $2\leqslant q\leqslant 6$, $\frac{2}{p}+\frac{3}{q}=\frac{3}{2}$, we have
	\begin{equation}
		\|\bm{u}\|_{L^{p}\left(-r_{n}^{2}, 0 ;L^{q}\left(U_{n}(R)\right)\right)}^{2} + \|\bm{b}\|_{L^{p}\left(-r_{n}^{2}, 0 ;L^{q}\left(U_{n}(R)\right)\right)}^{2} \leqslant C E_{n}(R).
	\end{equation}
\end{lem}

\section{A Key Lemma}\label{Sec 3}

In this section, we apply a similar argument in  Cafferalli-Kohn-Nirenberg \cite{Caffarelli1982}, Chae-Wolf \cite{Chae2021}, or Wang-Wu-Zhang \cite{Wang2020} to prove  the following lemma.
\begin{lem}\label{thm3}
	Let $(\bm{u},\bm{b}, \pi)$ be a suitable weak solution of \eqref{MHD} in $\R^3 \times(-1,0) .$ If $\bm{u},\bm{b}$ satisfies
	\begin{equation}\label{R3}
		u_{3} \in L^{p_{0},1}\left(-1,0;L^{q_{0}}(B(2))\right),\ b_{3} \in L^{p_{1},1}\left(-1,0;L^{q_{1}}(B(2))\right),
	\end{equation}
	with $\frac{2}{p_{0}}+\frac{3}{q_{0}}=\frac{2}{p_{1}}+\frac{3}{q_{1}}=1,~3 < q_{0},q_{1}<+ \infty$, then
	\begin{align*}
		r^{-2}\|\bm{u}\|_{L^{3}\left(Q(r) \right)}^{3}+r^{-2}\|\bm{b}\|_{L^{3}\left(Q(r)\right)}^{3} \leqslant C,
	\end{align*}
	for any $0<r\leqslant 1$. 
\end{lem}
For fixed $p_{0},p_{1}$, we can pick $\frac{4q_{0}}{2q_{0}-5}< p_{0}^*<p_{0}$ and $2<p_{1}^*<p_{1}$. Set
\begin{align}\label{3.2}
	\B_{i}=r_{i}^{1-\frac{2}{p_{0}^*}-\frac{3}{q_{0}}}\| u_{3}\|_{L^{p_{0}^*}\left(-r_{i}^2,0;L^{q_{0}}(B(2))\right)}+r_{i}^{1-\frac{2}{p_{1}^*}-\frac{3}{q_{1}}}\| b_{3}\|_{L^{p_{1}^*}\left(-r_{i}^2,0;L^{q_{1}}(B(2))\right)}.
\end{align}
By Lemma \ref{A5}, we have
\begin{align}\label{B}
	\sum_{i=0}^{+\infty} \B_{i}\leqslant C\left( \left\|u_{3}\right\|_{L^{p_{0}, 1}\left(-1,0;L^{q_{0}}(B(2))\right) }+\left\|b_{3}\right\|_{L^{p_{1}, 1}\left(-1,0;L^{q_{1}}(B(2))\right) }\right).
\end{align}
Let $\eta\left(x_{3}, t\right) \in C_{c}^{\infty}((-1,1) \times(-1,0])$ denotes a cut-off function,  $0 \leqslant \eta \leqslant 1$, and $\eta=1$ on $\left(-\frac{1}{2}, \frac{1}{2}\right) \times\left(-\frac{1}{4}, 0\right)$.

In addition, let $\frac{1}{2} \leqslant \rho<R\leqslant1$ be arbitrarily chosen, but $|R-\rho|\leqslant\frac{1}{2} .$ Let $\psi=\psi\left(x^{\prime}\right) \in C^{\infty}\left(\R^{2}\right)$ with $0 \leqslant \psi \leqslant 1$ in $B^{\prime}(R)$ satisfying
\begin{equation}
	\psi(x^{\prime})=\psi(|x^{\prime}|)=\left\{\begin{array}{l}
		1 \text { in } B^{\prime}(\rho) \\
		0 \text { in } \R^{2} \backslash B^{\prime}\left(\frac{R+\rho}{2}\right)
	\end{array}\right.,
\end{equation}
and
\begin{equation*}
	|D \psi| \leqslant \frac{C}{R-\rho}, \quad\left|D^{2} \psi\right| \leqslant \frac{C}{(R-\rho)^{2}}.
\end{equation*}

For $j=0,1,\cdots,n$, denote $\chi_{j} =\chi_{B^{\prime}(R)}(x^{\prime})\cdot\eta(2^{j}\cdot x_{3},2^{2j}\cdot t)$, where $\chi_{B^{\prime}(R)}$ is the indicator function of the set $B^{\prime}(R)$. Let
\begin{equation}\label{phi}
	\begin{array}{lll}
		\phi_{j}=\left\{\begin{array}{ll}
			\chi_{j}-\chi_{j+1}, & \text { if } \quad j=0, \ldots, n-1, \\
			\chi_{n}, & \text { if } \quad j=n.
		\end{array}\right.
	\end{array}
\end{equation}

Taking the test function $\varphi=\mathrm{\Phi}_n \eta\psi$ in \eqref{local energy}, we have
\begin{align}\label{key}
	&\int_{U_{0}(R)}|\left(\bm{u},\bm{b}\right)|^{2}\ \mathrm{\Phi}_{n} \eta \psi \dx+2 \int_{-1}^{t} \int_{U_{0}(R)}|\nabla \left(\bm{u},\bm{b}\right)|^{2}\ \mathrm{\Phi}_{n} \eta \psi \dx\ds \notag\\
	\leqslant& \int_{-1}^{t} \int_{U_{0}(R)}|\left(\bm{u},\bm{b}\right)|^{2}\cdot\left(\partial_{s} +\Delta \right)\left(\mathrm{\Phi}_{n} \eta \psi\right)\dx\ds\notag\\
	&+\int_{-1}^{t} \int_{U_{0}(R)}|\left(\bm{u},\bm{b}\right)|^2\ \bm{u} \cdot \nabla \left(\mathrm{\Phi}_{n} \eta \psi\right) -2\left(\bm{u}\cdot\bm{b}\right)\bm{b}\cdot\nabla\left(\mathrm{\Phi}_{n} \eta \psi\right)\dx\ds\notag\\
	&+2\int_{-1}^{t} \int_{U_{0}(R)} \pi \bm{u} \cdot \nabla\left(\mathrm{\Phi}_{n} \eta \psi\right)\dx\ds  .
\end{align}
Next, we shall handle the right side of \eqref{key} term by term.
\subsection{Estimates for nonlinear terms}\label{section3.1}
\begin{lem}\label{lem1}
	Under the assumptions of Lemma \ref{thm3}, we have
	\begin{equation}\label{2.5}
		\int_{-1}^{t} \int_{U_{0}(R)}|\left(\bm{u},\bm{b}\right)|^{2}\cdot\left(\partial_{s}+\Delta\right)\left(\mathrm{\Phi}_{n} \eta \psi\right)\dx\ds  \leqslant C \frac{\E}{(R-\rho)^{2}}.
	\end{equation}
\end{lem}
\begin{proof}
	By the estimates \eqref{Phi_{n}} for $\mathrm{\Phi}_{n}$ and the fact that $supp\  \nabla \eta \subseteq A_{0}, supp\  \phi_{j}\subseteq Q_{j}\backslash Q_{j+2}$, we have
	\begin{align*}
		&\int_{-1}^{t} \int_{U_{0}(R)}|\left(\bm{u},\bm{b}\right)|^{2}\cdot\left(\partial_{s}+\Delta\right)\left(\mathrm{\Phi}_{n} \eta \psi\right)\dx\ds\\
		=&\int_{-1}^{t} \int_{U_{0}(R)}|\left(\bm{u},\bm{b}\right)|^{2}\cdot\left(\mathrm{\Phi}_{n} \partial_{s} \eta \psi+2 \partial_{3} \mathrm{\Phi}_{n} \partial_{3} \eta \psi+\mathrm{\Phi}_{n} \Delta(\eta \psi)\right) \dx\ds\\
		\lesssim&  \int_{-1}^{0}\int_{U_{0}(R)}|\left(\bm{u},\bm{b}\right)|^{2} \dx\ds+ (R-\rho)^{-2}\sum_{i=0}^{n} \int_{-1}^{0}\int_{U_{0}(R)}|\left(\bm{u},\bm{b}\right)|^{2}\ \mathrm{\Phi}_{n} \phi_{i} \dx\ds \\
		\leqslant&\E+ (R-\rho)^{-2} \sum_{i=0}^{n} r_{i}^{-1} \int_{-r_{i}^2}^{0}\int_{U_{i}(R)}|\left(\bm{u},\bm{b}\right)|^{2} \dx\ds\\
		\lesssim&\E+(R-\rho)^{-2} \sum_{i=0}^{n} r_{i} \E \\
		\lesssim& (R-\rho)^{-2} \E.
	\end{align*}\qed
\end{proof}
\begin{lem}\label{lem2}
	Under the assumptions of Lemma \ref{thm3}, we have
	\begin{align}\label{2.6}
		&\int_{-1}^{t} \int_{U_{0}(R)}|\left(\bm{u},\bm{b}\right)|^2\ \bm{u} \cdot \nabla \left(\mathrm{\Phi}_{n} \eta \psi\right) -2\left(\bm{u}\cdot\bm{b}\right)\bm{b}\cdot\nabla\left(\mathrm{\Phi}_{n} \eta \psi\right)\dx\ds   \notag\\
		\leqslant& C\sum_{i=0}^{n} \B_{i}  \left(r_{i}^{-1} E_{i}(R)\right)+C\E^{\frac{3}{2}}+C(R-\rho)^{-1} \ \E^{\frac{1}{2}} \sum_{i=0}^{n} r_{i}^{\frac{1}{2}}\left(r_{i}^{-1} E_{i}(R)\right),
	\end{align}
	where $\B_{i}$ is defined in \eqref{3.2}.
\end{lem}
\begin{proof}
	We first note  that
	\begin{align}\label{I}
		&\int_{-1}^{t} \int_{U_{0}(R)}|\left(\bm{u},\bm{b}\right)|^2\ \bm{u} \cdot \nabla\left(\mathrm{\Phi}_{n} \eta \psi\right)-2\left(\bm{u}\cdot\bm{b}\right)\bm{b}\cdot\nabla\left(\mathrm{\Phi}_{n} \eta \psi\right)\dx\ds   \notag\\
		=&\int_{-1}^{t} \int_{U_{0}(R)}|\left(\bm{u},\bm{b}\right)|^2\ u_{3}\cdot\partial_{3} \mathrm{\Phi}_{n}\cdot \left(\eta \psi\right)\dx  \ds  \notag\\
		& +\int_{-1}^{t} \int_{U_{0}(R)}|\left(\bm{u},\bm{b}\right)|^2\ \bm{u} \cdot\mathrm{\Phi}_{n} \nabla\left(\eta\psi\right)\dx  \ds  \notag\\
		&-2\int_{-1}^{t} \int_{U_{0}(R)}\left(\bm{u}\cdot\bm{b}\right)b_{3}\cdot\partial_{3}\mathrm{\Phi}_{n}\cdot \left(\eta \psi\right)\dx\ds\notag\\
		&-2\int_{-1}^{t} \int_{U_{0}(R)}\left(\bm{u}\cdot\bm{b}\right)\bm{b}\cdot\mathrm{\Phi}_{n} \nabla\left(\eta \psi\right)\dx\ds \notag\\
		\eqdefa &I_{1}+I_{2}+I_{3}+I_{4}.
	\end{align}
	
	By the estimates \eqref{Phi_{n}} for $\mathrm{\Phi}_{n}$, H\"{o}lder inequality and Lemma \ref{energy}, we have
	\begin{align}\label{I_{1}}
		I_{1}
		=&\sum_{i=0}^{n}  \int_{-1}^{t} \int_{U_{0}(R)}|\left(\bm{u},\bm{b}\right)|^2\  u_{3}\cdot\partial_{3} \mathrm{\Phi}_{n}\cdot \left(\phi_{i} \psi\right)\dx  \ds  \notag\\
		\lesssim&\sum_{i=0}^{n}  r_{i}^{-2} \int_{-r_{i}^2}^{0}\|\left(\bm{u},\bm{b}\right)\|_{L^{\frac{2q_{0}}{q_{0}-1}}(U_{i}(R))}^2\|u_{3}\|_{L^{q_{0}}(U_{i}(R))} \ds  \notag\\
		\lesssim& \sum_{i=0}^{n} r_{i}^{-\frac{2}{p_{0}^*}-\frac{3}{q_{0}}}\|\left(\bm{u},\bm{b}\right)\|_{L^{\frac{4q_{0}}{3}}\left(-r_{i}^2,0;L^{\frac{2q_{0}}{q_{0}-1}}(U_{i}(R))\right)}^2 \|u_{3}\|_{L^{p_{0}^*}\left(-r_{i}^2,0;L^{q_{0}}(U_{i}(R))\right)}\notag\\
		\lesssim&\sum_{i=0}^{n} \B_{i} \left(r_{i}^{-1} E_{i}(R)\right).
	\end{align}
	For the second term, by \eqref{Phi_{n}} and Lemma \ref{energy} again, we have
	\begin{align}\label{I_{2}}
		I_{2} =&\int_{-1}^{t} \int_{U_{0}(R)}|\left(\bm{u},\bm{b}\right)|^2\ \bm{u} \cdot\mathrm{\Phi}_{n} \nabla\eta\psi+|\left(\bm{u},\bm{b}\right)|^2\ \bm{u} \cdot\mathrm{\Phi}_{n} \eta\nabla\psi\dx  \ds\notag\\
		\lesssim&\int_{-1}^{0}\int_{U_{0}(R)}|\left(\bm{u},\bm{b}\right)|^3\dx\ds+ \frac{1}{(R-\rho)} \sum_{i=0}^{n} r_{i}^{-1} \int_{-r_{i}^2}^{0}\int_{U_{i}(R)}|\left(\bm{u},\bm{b}\right)|^{3} \dx \ds \notag\\
		\lesssim&\E^{\frac{3}{2}}+ \frac{1}{(R-\rho)} \sum_{i=0}^{n} r_{i}^{-1} r_{i}^{\frac{1}{2}}\ \|\left(\bm{u},\bm{b}\right)\|_{L^{4}\left(-r_{i}^{2}, 0 ; L^{3}\left(U_{i}(R)\right)\right)}^{3} \notag\\
		\lesssim&\E^{\frac{3}{2}}+\frac{\E^{\frac{1}{2}}}{(R-\rho)}  \sum_{i=0}^{n} r_{i}^{\frac{1}{2}} \left(r_{i}^{-1} E_{i}(R)\right).
	\end{align}
	As to $I_{3},I_{4}$ , we have
	\begin{align}\label{I_{3}}
		I_{3}
		\lesssim&\sum_{i=0}^{n}  r_{i}^{-2} \int_{-r_{i}^2}^{0}\ \|\left(\bm{u},\bm{b}\right)\|_{L^{\frac{2q_{1}}{q_{1}-1}}(U_{i}(R))}^2 \|b_{3}\|_{L^{q_{1}}(U_{i}(R))} \ds  \notag\\
		\lesssim& \sum_{i=0}^{n} r_{i}^{-\frac{2}{p_{1}^*}-\frac{3}{q_{1}}}\ \|\left(\bm{u},\bm{b}\right)\|_{L^{\frac{4q_{1}}{3}}\left(-r_{i}^2,0;L^{\frac{2q_{1}}{q_{1}-1}}(U_{i}(R))\right)}^2 \|b_{3}\|_{L^{p_{1}^*}\left(-r_{i}^2,0;L^{q_{1}}(U_{i}(R))\right)}\notag\\
		\lesssim&\sum_{i=0}^{n} \B_{i} \left(r_{i}^{-1} E_{i}(R)\right),
	\end{align}
	which is analogous to \eqref{I_{1}}, and
	\begin{align}\label{I_{4}}
		I_{4}\lesssim&\int_{-1}^{0}\int_{U_{0}(R)}|\left(\bm{u},\bm{b}\right)|^3\dx\ds+ \frac{1}{(R-\rho)} \sum_{i=0}^{n} r_{i}^{-1} \int_{-r_{i}^2}^{0}\int_{U_{i}(R)}|\left(\bm{u},\bm{b}\right)|^{3} \dx \ds \notag\\
		\lesssim&\E^{\frac{3}{2}}+\frac{\E^{\frac{1}{2}}}{(R-\rho)}  \sum_{i=0}^{n} r_{i}^{\frac{1}{2}}\left(r_{i}^{-1} E_{i}(R)\right),
	\end{align}
	which is analogous to \eqref{I_{2}}.
	Combining \eqref{I} with \eqref{I_{1}}, \eqref{I_{2}}, \eqref{I_{3}} and \eqref{I_{4}},  we obtain \eqref{2.6}.
	\qed
\end{proof}
\subsection{Estimates for the pressure}\label{section3.2}

This part is devoted to the estimates regarding the
pressure term on the right side of \eqref{key}. We will give a similar argument in \cite{Wang2020} here. 

Given $3\times3$ matrix valued function $f=\left(f_{jk}\right)$, we set
\begin{align*}
	\widehat{\J(f)}(\xi)=\sum_{j,k=1,2,3} \frac{\xi_{j}\xi_{k}}{|\xi|^2}\mathcal{F} f_{jk},
\end{align*}
where the Fourier transform is defined by
\begin{equation*}
	\mathcal{F} g(\xi)=\hat{g}(\xi)=\int_{\R^3} e^{-i(x \cdot \xi)} g(x)\dx  .
\end{equation*}
Therefore, $\J: L^{q}\left(\R^3\right) \rightarrow L^{q}\left(\R^{3}\right)$,  $1<q<+\infty$ defines a bounded linear operator with
\begin{align}\label{ani}
	\|\J(f)\|_{L^{q}\left(\R^{3}\right)} \leqslant C\ \|f\|_{L^{q}\left(\R^3\right)}.
\end{align}

Denote $f=(f_{jk})=\left(\bm{u}\otimes\bm{u}-\bm{b}\otimes\bm{b}\right)\cdot \chi_{0}$,
\begin{align}\label{pitau}
	\pi_{0}=\J\left(f\right),\ \pi_{h}=\pi-\pi_{0}.
\end{align}
It follows that
\begin{align*}
	\Delta \pi_{0}=\nabla \cdot  \nabla \cdot f=\mathrm{\Sigma}_{j,k=1}^{3}\  \partial_{j} \partial_{k} f_{jk} \quad \text { in } \quad \R^{3} \times(-1,0)
\end{align*}
in the sense of distributions and  $\pi_{h}$  is harmonic in $Q_{1}(R)$. Then we have
\begin{align}\label{pressure}
	&\int_{-1}^{t} \int_{U_{0}(R)} \pi \bm{u} \cdot \nabla\left(\mathrm{\Phi}_{n} \eta \psi\right)\dx  \ds   \notag\\
	=&\int_{-1}^{t} \int_{U_{0}(R)} \pi_{0} \bm{u} \cdot \nabla\left(\mathrm{\Phi}_{n} \eta \psi\right)\dx  \ds  +\int_{-1}^{t} \int_{U_{0}(R)} \pi_{h} \bm{u} \cdot \nabla\left(\mathrm{\Phi}_{n} \eta \psi\right)\dx  \ds   \notag\\
	=&\int_{-1}^{t} \int_{U_{0}(R)} \pi_{0} u_{3} \cdot \partial_{3}\mathrm{\Phi}_{n} \eta \psi\dx  \ds  +\int_{-1}^{t} \int_{U_{0}(R)} \pi_{0} \bm{u} \cdot \mathrm{\Phi}_{n} \nabla\left(\eta \psi\right)\dx  \ds  \notag\\
	&-\int_{-1}^{t} \int_{U_{0}(R)} \nabla\pi_{h}\cdot \bm{u}\cdot  \left(\mathrm{\Phi}_{n} \eta \psi\right)\dx  \ds  \notag\\
	\eqdefa& J+K+H.
\end{align} 
\begin{lem}\label{lem3}  Under the assumptions of Lemma \ref{thm3}, we have
	\begin{align}\label{p1}
		&\int_{-1}^{t} \int_{U_{0}(R)} \pi \bm{u} \cdot \nabla\left(\mathrm{\Phi}_{n} \eta \psi\right)\dx  \normalfont{\ds}  \notag\\
		\leqslant& C\sum_{i=0}^{n} \B_{i} \left(r_{i}^{-1} E_{i}(R)\right)+C\frac{\E^{\frac{1}{2}}}{R-\rho}  \sum_{i=0}^{n} r_{i}^{\frac{1}{2}}\left(r_{i}^{-1} E_{i}(R)\right)+\frac{C}{(R-\rho)^{3}} \E^{\frac{3}{2}}.
	\end{align}
\end{lem}
\begin{proof}
	Combining the estimates \eqref{pressure}, \eqref{J}, \eqref{K} and \eqref{H}, we have \eqref{p1}.\qed
\end{proof}
\subsection{The proof of Lemma \ref{thm3}}\label{section3.3}
On the basis of the estimates of the nonlinear term and the pressure in subsection 3.1--3.2, we are in position to give the detail proof of  Lemma \ref{thm3}.
\begin{proof}
	Gathering \eqref{key} and the estimates in Lemma \ref{lem1}, \ref{lem2} and  \ref{lem3}, we have
	\begin{align*}
		r_{n}^{-1} E_{n}(\rho)
		\leqslant& C\sum_{i=0}^{n} \B_{i}  \left(r_{i}^{-1} E_{i}(R)\right)+ C\frac{ \E^{\frac{1}{2}}}{R-\rho} \sum_{i=0}^{n} r_{i}^{\frac{1}{2}}\left(r_{i}^{-1} E_{i}(R)\right)+C\frac{1+\E^{\frac{3}{2}}}{(R-\rho)^{3}}\\
		\leqslant& C\sum_{i=0}^{n} \B_{i}  \left(r_{i}^{-1} E_{i}(R)\right)+C\frac{\E^{\frac{3}{4}}}{R-\rho}\sum_{i=0}^{n}r_{i}^{-\frac{5}{8}}E_{i}(R)^{\frac{3}{4}}r_{i}^{\frac{1}{8}}+C\frac{1+\E^{\frac{3}{2}}}{(R-\rho)^{3}}\\
		\leqslant &  C\sum_{i=0}^{n} \B_{i}  \left(r_{i}^{-1} E_{i}(R)\right)+C\frac{\E^{\frac{3}{4}}}{R-\rho}\left(\sum_{i=0}^{n}r_{i}^{-\frac{5}{6}}E_{i}(R)\right)^{\frac{3}{4}}\left(\sum_{i=0}^{n}r_{i}^{\frac{1}{2}}\right)^{\frac{1}{4}}\\
		&+C\frac{1+\E^{\frac{3}{2}}}{(R-\rho)^{3}}\\
		\leqslant&C_{0}\sum_{i=0}^{n} \left(\B_{i}+r_{i}^{\frac{1}{6}}\right)  \left(r_{i}^{-1} E_{i}(R)\right)+C_{0}\frac{1+\E^{3}}{(R-\rho)^{4}}.
	\end{align*}
	In view of \eqref{B}, there exists a sufficient large integer $n_{0}\geqslant 1$ such that
	\begin{equation}\label{small}
		C_{0}\sum_{i=n_{0}}^{\infty} \left(\B_{i}+r_{i}^{\frac{1}{6}}\right)\leqslant \frac{1}{2}.
	\end{equation}
	Then for $n\geqslant n_{0} $ we have
	\begin{align}\label{iteration1}
		r_{n}^{-1} E_{n}(\rho) \leqslant& C_{0}\sum_{i=n_{0}}^{n} \left(\B_{i}+r_{i}^{\frac{1}{6}}\right)  \left(r_{i}^{-1} E_{i}(R)\right)\nonumber\\
		&\quad+C_{0}\sum_{i=0}^{n_{0}-1} \left(\B_{i}+r_{i}^{\frac{1}{6}}\right)  \left(r_{i}^{-1} E_{i}(R)\right)+C_{0}\frac{1+\E^{3}}{(R-\rho)^{4}}\notag\\
		\leqslant& C_{0}\sum_{i=n_{0}}^{n} \left(\B_{i}+r_{i}^{\frac{1}{6}}\right)  \left(r_{i}^{-1} E_{i}(R)\right)+\frac{A_{0}}{(R-\rho)^{4}},
	\end{align}
	where the constant  
	$$A_{0}=C_{1}\cdot2^{n_{0}} \left(1+ \left\|u_{3}\right\|_{L^{p_{0}, 1}\left(-1,0;L^{q_{0}}(B(2))\right) }^2+\left\|b_{3}\right\|_{L^{p_{1}, 1}\left(-1,0;L^{q_{1}}(B(2))\right)}^2+\E^{3}\right).$$
	
	As the iteration argument in \cite[V. Lemma 3.1]{Giaquinta1983}, we introduce the sequence $\{\rho_{k}\}_{k=0}^{+\infty}$ which satisfy
	\begin{equation*}
		\rho_{0}=\frac{1}{2},\quad \rho_{k+1}-\rho_{k}=\frac{1-\theta}{2}\theta^{k},
	\end{equation*}
	with  $\frac{1}{2}<\theta^4<1$ and $\lim_{k\rightarrow\infty}\rho_{k}=1$. For $n_{0}\leqslant j\leqslant n$ and $k\geqslant 1$, we have
	\begin{align}\label{iteration2}
		r_{j}^{-1} E_{j}(\rho_{k})\leqslant& C_{0}\sum_{i=n_{0}}^{j} \left(\B_{i}+r_{i}^{\frac{1}{6}}\right)  \left(r_{i}^{-1} E_{i}(\rho_{k+1})\right)+A_{0}\cdot \frac{16}{(1-\theta)^4}\theta^{-4k}\notag\\
		\leqslant&C_{0}\sum_{i=n_{0}}^{n} \left(\B_{i}+r_{i}^{\frac{1}{6}}\right)  \left(r_{i}^{-1} E_{i}(\rho_{k+1})\right)+A_{0}\cdot \frac{16}{(1-\theta)^4}\theta^{-4k}.
	\end{align}
	Applying \eqref{small} and \eqref{iteration2} to the  iteration argument from \eqref{iteration1}, we obtain that for $n\geqslant n_{0}$,
	\begin{align*}
		r_{n}^{-1} E_{n}(\rho_{0}) \leqslant&C_{0}\sum_{j=n_{0}}^{n} \left(\B_{j}+r_{j}^{\frac{1}{6}}\right)  \left(r_{j}^{-1} E_{j}(\rho_{1})\right)+A_{0}\cdot \frac{16}{(1-\theta)^4}\\
		\leqslant&\frac{1}{2} C_{0}\sum_{j=n_{0}}^{n} \left(\B_{j}+r_{j}^{\frac{1}{6}}\right)  \left(r_{j}^{-1} E_{j}(\rho_{2})\right)+A_{0}\cdot \frac{16}{(1-\theta)^4}\left(1+\frac{1}{2\theta^4}\right)\\
		\leqslant& \frac{1}{2^{k-1}}C_{0}\sum_{j=n_{0}}^{n} \left(\B_{j}+r_{j}^{\frac{1}{6}}\right)  \left(r_{j}^{-1} E_{j}(\rho_{k})\right)+A_{0}\cdot \frac{16}{(1-\theta)^4}\sum_{j=0}^{k-1} \left(\frac{1}{2\theta^4}\right)^j\\
		\leqslant&\frac{1}{2^{k-1}}C_{0}\sum_{j=n_{0}}^{n} \left(\B_{j}+r_{j}^{\frac{1}{6}}\right)  \left(r_{j}^{-1} E_{j}(1)\right)+A_{0}\cdot \frac{16}{(1-\theta)^4}\cdot \frac{2\theta^4}{2\theta^4-1}.
	\end{align*}
	Let $k\rightarrow\infty$, we obtain that for $n\geqslant n_{0}$,
	\begin{equation}\label{3.38}
		r_{n}^{-1} E_{n}(\rho_{0}) \leqslant A_{0}\cdot \frac{16}{(1-\theta)^4}\cdot \frac{2\theta^4}{2\theta^4-1}.
	\end{equation}
	For $0<r\leqslant r_{n_{0}}$, there exists an integer $j\geqslant n_{0}$, such that
	\begin{equation*}
		r_{j+1}<r\leqslant r_{j},
	\end{equation*}
	which together with \eqref{3.38} ensures that
	\begin{align*}
		r^{-1}\sup_{t\in (-r^2, 0)}\int_{B(r)}|\left(\bm{u},\bm{b}\right)|^2\dx
		+r^{-1}\int_{Q(r)}|\nabla\left(\bm{u},\bm{b}\right)|^2\dx\dt
		\leqslant 2\ r_{j}^{-1} E_{j}(\rho_{0})\leqslant C .
	\end{align*}
	For $r_{n_{0}}<r \leqslant 1$,
	\begin{align*}
		r^{-1}\sup_{t\in (-r^2, 0)}\int_{B(r)}|\left(\bm{u},\bm{b}\right)|^2\dx
		+r^{-1}\int_{Q(r)}|\nabla\left(\bm{u},\bm{b}\right)|^2\dx\dt\leqslant 2^{n_{0}}\E\leqslant C.
	\end{align*}
	Hence for $0<r\leqslant 1$,
	\begin{align}\label{3.22}
		r^{-1}\sup_{t\in (-r^2, 0)}\int_{B(r)}|\left(\bm{u},\bm{b}\right)|^2\dx
		+r^{-1}\int_{Q(r)}|\nabla\left(\bm{u},\bm{b}\right)|^2\dx\dt\leqslant
		C.
	\end{align}
	By Sobolev embedding theorem, we have that for $0<r\leqslant 1$,
	\begin{align}
		r^{-2}\|\left(\bm{u},\bm{b}\right)\|_{L^{3}(Q(r))}^{3}
		\leqslant C r^{-\frac{3}{2}}\|\left(\bm{u},\bm{b}\right)\|_{L^{4}\left(-r^{2}, 0 ; L^{3}(B(r))\right)}^{3} \leqslant C.
	\end{align}
	The proof is completed.\qed
\end{proof}
\section{Proof of Theorem \ref{thm1}}\label{Sec 4}
Before starting the proof, we give an essential lemma for the 3D MHD equation \eqref{MHD}.
\begin{lem}\label{A8}
	Let $\mu=\nu=1$ in the MHD equations \eqref{MHD}. Assume that $\left(\bm{u},\bm{b},\pi\right)$ is a suitable weak solution in $Q(1)$ with $u_{3}=b_{3}=0$. Then the solution is  regular in $\overline{Q(\frac{1}{2})}$. 
\end{lem}
\begin{proof}
	Denote $\mathcal{S}_{\mathrm{MHD}}$ the singular set of the suitable weak solution, the cylinder $U(x,r)=B^{\prime}(x^{\prime},r)\times(x_{3}-r,x_{3}+r)$ and $U(r):=U(0,r)$. By \cite{He2005} or \cite[Theorem 1.4]{Vyalov2008}, we know that the one-dimensional Hausdorff measure of $\mathcal{S}_{\mathrm{MHD}}$ is zero. 
	
	Assume that the assertion does not hold, i.e. $\mathcal{S}_{\mathrm{MHD}} \cap \overline{Q(\frac{1}{2})}\neq \varnothing$. By a similar argument in \cite[Lemma 3.2]{kukavica2017}, there exists $(x_{0},t_{0})\in\overline{Q(\frac{1}{2})}$ and some $\varepsilon,\rho$ with $0<\varepsilon<\rho<\frac{1}{2}$, such that
	\begin{align*}
		(x_{0},t_{0})\in\mathcal{S}_{\mathrm{MHD}} \cap \overline{U(x_{0},\rho)\times(t_{0}-\rho^2,t_{0})} \subseteq U\left(x_{0},\rho-\varepsilon\right) \times\left\{t_{0}\right\}.
	\end{align*}
	Without loss of generality, we may assume $(x_{0},t_{0})=(0,0)$. Thus, the solution  $(\bm{u},\bm{b},\nabla\pi)\in C^{\infty}(\overline{U(\rho)}\times[-\rho^2,0))$ and is singular at $(0,0)$. Moreover, there exists a constant $C$, such that 
	\begin{align}\label{A.21}
		\sup_{-1<t<0}\|\left(\bm{u},\bm{b}\right)\|_{L^{2}(B(1))}^2+\|\nabla\left(\bm{u},\bm{b}\right)\|_{L^{2}(Q(1))}^2+\|\pi\|_{L^{\frac{3}{2}}(Q(1))}\leqslant C,
	\end{align}
	and for $(x,t)\in U(\rho)\backslash U(\rho-\varepsilon)\times[-\rho^2,0)$ or $(x,t)\in U(\rho)\times\{-\rho^2\}$,
	\begin{align}\label{A.22}
		|\bm{u}|+|\bm{b}|\leqslant C.
	\end{align}
	
	We introduce the Els\"{a}sser variables  $\bm{Z}^{+}=\left(Z_{1}^{+},Z_{2}^{+},0\right)=\bm{u}+\bm{b}$ and $\bm{Z}^{-}=\left(Z_{1}^{-},Z_{2}^{-},0\right)=\bm{u}-\bm{b}$, which satisfy
	\begin{align*}
		&\partial_{t}\bm{Z}^{+}+\bm{Z}^{-}\cdot\nabla\bm{Z}^{+}-\Delta  \bm{Z}^{+}+\nabla \pi=0,\\
		&\partial_{t}\bm{Z}^{-}+\bm{Z}^{+}\cdot\nabla\bm{Z}^{-}-\Delta\bm{Z}^{-}+\nabla \pi=0.
	\end{align*}
	
	As in \cite{Cao2007}, we introduce
	$\overline{\bm{Z}^{+}}=\int_{-\rho}^{\rho}\bm{Z}^{+}\cdot\beta(x_{3})\dx_{3},\ \overline{\bm{Z}^{-}}=\int_{-\rho}^{\rho}\bm{Z}^{-}\cdot\beta(x_{3})\dx_{3}$, where the non-negative cut-off function $\beta(x_{3})\in C_{c}^{\infty}((-\rho,\rho))$, $\beta=\frac{1}{2\rho-\varepsilon}$ on $(-\rho+\varepsilon,\rho-\varepsilon)$ and $\int_{-\rho}^{\rho}\beta(x_{3})\dx_{3}=1$. It is easy to find that $\partial_{3}\pi=0$, and
	\begin{align*}
		\partial_{t} \overline{\bm{Z}^{+}}+\int_{-\rho}^{\rho}\bm{Z}^{-}\cdot\nabla\bm{Z}^{+}\cdot\beta\dx_{3}-\Delta \overline{\bm{Z}^{+}}+\nabla\pi
		&=\int_{-\rho}^{\rho}\bm{Z}^{+}\cdot\partial_{3}^2\beta\dx_{3},\\
		\partial_{t} \overline{\bm{Z}^{-}}+\int_{-\rho}^{\rho}\bm{Z}^{+}\cdot\nabla\bm{Z}^{-}\cdot\beta\dx_{3}-\Delta \overline{\bm{Z}^{-}}+\nabla\pi
		&=\int_{-\rho}^{\rho}\bm{Z}^{-}\cdot\partial_{3}^2\beta\dx_{3}.
	\end{align*}
	
	Define $\widetilde{\bm{Z}^{+}}=\bm{Z}^{+}-\overline{\bm{Z}^{+}}$, $ \widetilde{\bm{Z}^{-}}=\bm{Z}^{-}-\overline{\bm{Z}^{-}}$, which satisfy
	\begin{align}
		\partial_{t} \widetilde{\bm{Z}^{+}}+\bm{Z}^{-}\cdot\nabla\widetilde{\bm{Z}^{+}}+\widetilde{\bm{Z}^{-}}\cdot\nabla\overline{\bm{Z}^{+}}-\int_{-\rho}^{\rho}\widetilde{\bm{Z}^{-}}\cdot\nabla\widetilde{\bm{Z}^{+}}\cdot\beta\dx_{3}-\Delta \widetilde{\bm{Z}^{+}}
		=&-\int_{-\rho}^{\rho}\bm{Z}^{+}\cdot\partial_{3}^2\beta\dx_{3}\label{A.27}\\
		\partial_{t} \widetilde{\bm{Z}^{-}}+\bm{Z}^{+}\cdot\nabla\widetilde{\bm{Z}^{-}}+\widetilde{\bm{Z}^{+}}\cdot\nabla\overline{\bm{Z}^{-}}-\int_{-\rho}^{\rho}\widetilde{\bm{Z}^{+}}\cdot\nabla\widetilde{\bm{Z}^{-}}\cdot\beta\dx_{3}-\Delta \widetilde{\bm{Z}^{-}}
		=&-\int_{-\rho}^{\rho}\bm{Z}^{-}\cdot\partial_{3}^2\beta\dx_{3}.\label{A.28}
	\end{align}
	
	Let $ \alpha(x_{1},x_{2})\in C_{c}^{\infty}(B^{\prime}(\rho))$ be a cut-off function with $0\leqslant \alpha\leqslant 1$, $\alpha=1$ on $B^{\prime}(\rho-\varepsilon)$. Multiplying \eqref{A.27} by $4|\widetilde{\bm{Z}^{+}}|^2\widetilde{\bm{Z}^{+}}\cdot\varphi^4$ with $\varphi(x)=\alpha(x_{1},x_{2})\cdot\beta(x_{3})$, and integrating the resulting equations over $\R^3$, applying \eqref{A.21}, \eqref{A.22}, H\"{o}lder's inequality and Sobolev embedding theorem, we have
	\begin{align}\label{A.29}
		&\frac{d}{dt}\|\widetilde{\bm{Z}^{+}}\varphi\|_{L^{4}(U(\rho))}^4+4\int_{\R^3}|\nabla\widetilde{\bm{Z}^{+}}|^2|\widetilde{\bm{Z}^{+}}|^2\varphi^4\dx+2\int_{\R^3}|\nabla\left(|\widetilde{\bm{Z}^{+}}|^2\varphi^2\right)|^2\dx\notag\\
		=&\int_{\R^3}\bm{Z}^{-}\cdot\nabla\alpha^4\cdot\beta^4|\widetilde{\bm{Z}^{+}}|^4-4\widetilde{\bm{Z}^{-}}\cdot\nabla\overline{\bm{Z}^{+}}\cdot|\widetilde{\bm{Z}^{+}}|^2\widetilde{\bm{Z}^{+}}\varphi^4\dx\notag\\
		&+4\int_{\R^3}\left(\int_{-\rho}^{\rho}\widetilde{\bm{Z}^{-}}\cdot\nabla\widetilde{\bm{Z}^{+}}\cdot\beta-\bm{Z}^{+}\cdot\partial_{3}^2\beta\dx_{3}\right)\cdot|\widetilde{\bm{Z}^{+}}|^2\widetilde{\bm{Z}^{+}}\varphi^4\dx+2\int_{\R^3}|\widetilde{\bm{Z}^{+}}|^4|\nabla\varphi^2|^2\dx\notag\\
		\lesssim&1+\left(\|\nabla\overline{\bm{Z}^{+}}\|_{L^{2}(B^{\prime}(\rho))}+\|\nabla\widetilde{\bm{Z}^{+}}\|_{L^{2}(U(\rho))}\right)\cdot\|\widetilde{\bm{Z}^{-}}\varphi\|_{L^{4}(U(\rho))}\|\widetilde{\bm{Z}^{+}}\varphi\|_{L^{4}(U(\rho))}\notag\\
		&\times\|\nabla\left(|\widetilde{\bm{Z}^{+}}|^2\varphi^2\right)\|_{L^{2}(U(\rho))}+\|\widetilde{\bm{Z}^{+}}\varphi\|_{L^{4}(U(\rho))}^3
		\notag\\
		\leqslant&C+C\left(1+\|\nabla\bm{Z}^{+}\|_{L^{2}(U(\rho))}^2\right)\cdot\left(\|\widetilde{\bm{Z}^{+}}\varphi\|_{L^{4}(U(\rho))}^4+\|\widetilde{\bm{Z}^{-}}\varphi\|_{L^{4}(U(\rho))}^4\right)\notag\\
		&-\|\nabla\left(|\widetilde{\bm{Z}^{+}}|^2\varphi^2\right)\|_{L^{2}(U(\rho))}^2.
	\end{align}
	Analogously, multiplying \eqref{A.28} by $4|\widetilde{\bm{Z}^{-}}|^2\widetilde{\bm{Z}^{-}}\cdot\varphi^4$, and integrating the resulting equations over $\R^3$, we have
	\begin{align}\label{A.30}
		&\frac{d}{dt}\|\widetilde{\bm{Z}^{-}}\varphi\|_{L^{4}(U(\rho))}^4+4\int_{\R^3}|\nabla\widetilde{\bm{Z}^{-}}|^2|\widetilde{\bm{Z}^{-}}|^2\varphi^4\dx+2\int_{\R^3}|\nabla\left(|\widetilde{\bm{Z}^{-}}|^2\varphi^2\right)|^2\dx\notag\\
		=&\int_{\R^3}\bm{Z}^{+}\cdot\nabla\alpha^4\cdot\beta^4|\widetilde{\bm{Z}^{-}}|^4-4\widetilde{\bm{Z}^{+}}\cdot\nabla\overline{\bm{Z}^{-}}\cdot|\widetilde{\bm{Z}^{-}}|^2\widetilde{\bm{Z}^{-}}\varphi^4\dx\notag\\
		&+4\int_{\R^3}\left(\int_{-\rho}^{\rho}\widetilde{\bm{Z}^{+}}\cdot\nabla\widetilde{\bm{Z}^{-}}\cdot\beta-\bm{Z}^{-}\cdot\partial_{3}^2\beta\dx_{3}\right)\cdot|\widetilde{\bm{Z}^{-}}|^2\widetilde{\bm{Z}^{-}}\varphi^4\dx+2\int_{\R^3}|\widetilde{\bm{Z}^{-}}|^4|\nabla\varphi^2|^2\dx\notag\\
		\leqslant&C+C\left(1+\|\nabla\bm{Z}^{-}\|_{L^{2}(U(\rho))}^2\right)\cdot\left(\|\widetilde{\bm{Z}^{+}}\varphi\|_{L^{4}(U(\rho))}^4+\|\widetilde{\bm{Z}^{-}}\varphi\|_{L^{4}(U(\rho))}^4\right)\notag\\
		&-\|\nabla\left(|\widetilde{\bm{Z}^{-}}|^2\varphi^2\right)\|_{L^{2}(U(\rho))}^2.
	\end{align}
	Summing up \eqref{A.29} and \eqref{A.30} and applying Gronwall's inequality, we have that for all $-\rho^2<t<0$,
	\begin{align}\label{A.31}
		\|\widetilde{\bm{Z}^{+}}\|_{L^{4}(U(\rho-\varepsilon))}^4+\|\widetilde{\bm{Z}^{-}}\|_{L^{4}(U(\rho-\varepsilon))}^4
		\leqslant C	\|\widetilde{\bm{Z}^{+}}\varphi\|_{L^{4}(U(\rho))}^4+C\|\widetilde{\bm{Z}^{-}}\varphi\|_{L^{4}(U(\rho))}^4\leqslant C.
	\end{align}
	Accordingly, by \eqref{A.22} and \eqref{A.31}, we have that for $0<r<\rho-\varepsilon$,
	\begin{align*}
		&r^{-2}\left(\|\bm{u}\|_{L^{3}(Q(r))}^3+\|\bm{b}\|_{L^{3}(Q(r))}^3\right)\\
		\leqslant& Cr^{-2}\left(\|\bm{Z}^{+}\|_{L^{3}(Q(r))}^3+\|\bm{Z}^{-}\|_{L^{3}(Q(r))}^3\right)\\
		\leqslant& Cr^{-2}\left(\|\widetilde{\bm{Z}^{+}}\|_{L^{3}(Q(r))}^3+\|\widetilde{\bm{Z}^{-}}\|_{L^{3}(Q(r))}^3\right)+Cr^{-2}\left(\|\overline{\bm{Z}^{+}}\|_{L^{3}(Q(r))}^3+\|\overline{\bm{Z}^{-}}\|_{L^{3}(Q(r))}^3\right)\\
		\leqslant & C r^{\frac{3}{4}}+Cr^{-1}\int_{-r^{2}}^{0}\int_{B^{\prime}(r)}\left(|\overline{\bm{Z}^{+}}|^3+|\overline{\bm{Z}^{-}}|^3\right)\dx_{1}\dx_{2}\dt\\
		\leqslant & C r^{\frac{3}{4}}+C\sup_{-1<t<0}\|\left(\bm{u},\bm{b}\right)\|_{L^{2}(B(1))}^2\cdot\left(\int_{-r^2}^{0}\int_{-\rho}^{\rho}\int_{B^{\prime}(r)}|\nabla \left(\bm{u},\bm{b}\right)|^2\dx\dt\right)^{\frac{1}{2}}.
	\end{align*}
	We can pick a sufficient small $0<R_{0}<1$, such that for all $0<r<R_{0}$,  
	\begin{align*}
		r^{-2}\left(\|\bm{u}\|_{L^{3}(Q(r))}^3+\|\bm{b}\|_{L^{3}(Q(r))}^3\right)<\varepsilon_{1}.
	\end{align*}
	By Lemma \ref{A7}, we have that the solution is regular at $(0,0)$, which leads to a contradiction.\qed
\end{proof}

Now, we are in position to prove Theorem \ref{thm1} by standard rigid methods.

\noindent For $r_{k}=2^{-k},\  k\geqslant1$ and  $(x, t) \in Q(1)$, we denote the following scaling
\begin{align*}
	&\bm{u}_{k}(x, t)=r_{k} \bm{u}\left(r_{k} x, r_{k}^{2} t\right),\ 		\bm{b}_{k}(x, t)=r_{k} \bm{b}\left(r_{k} x, r_{k}^{2} t\right) \\
	&\pi_{k}(x, t)=r_{k}^{2}\pi\left(r_{k} x, r_{k}^{2} t\right).
\end{align*}
By \eqref{3.22}, we obtain that for  $k\geqslant 1$, 
\begin{align}\label{4.4}
	\sup_{t\in (-1, 0)}\|(\bm{u}_{k},\bm{b}_{k})\|^2_{L^{2}(B(1))}
	+\|\nabla(\bm{u}_{k},\bm{b}_{k})\|^2_{L^{2}(Q(1))}\leqslant C.
\end{align}
Applying Lemma \ref{thm3} and Lemma \ref{A6}, we obtain that for  $0<r\leqslant 1$,
\begin{align*}
	r^{-2}\|\pi\|_{L^{\frac{3}{2}}(Q(r))}^{\frac{3}{2}}\leqslant C,
\end{align*}
which implies that 
\begin{align}\label{4.5}
	\|\pi_{k}\|_{L^{\frac{3}{2}}(Q(1))}\leqslant C.
\end{align}
Moreover, $\left(\bm{u}_{k},\bm{b}_{k},\pi_{k}\right)$ is a suitable weak solution to the MHD equations \eqref{MHD} in $Q(1)$, which converges weakly (by taking subsequences if needed) to some $\left(\bm{v},\bm{B},\Pi\right)$,
\begin{align}\label{control1}
	\sup_{t\in (-1, 0)}\|\left(\bm{v},\bm{B}\right)\|^2_{L^{2}(B(1))}+\|\nabla\left(\bm{v},\bm{B}\right)\|^2_{L^{2}(Q(1))}+\|\Pi\|_{L^{\frac{3}{2}}(Q(1))}\leqslant C.
\end{align}
By a similar argument as \cite[Theorem 2.2]{Lin1998}, we can prove that $\left(\bm{v},\bm{B},\Pi\right)$ is a suitable weak solution to the MHD equations in $Q(1)$ and
$\left(\bm{u}_{k},\bm{b}_{k}\right)\underset{k\rightarrow\infty}{\longrightarrow} \left(\bm{v},\bm{B}\right)$ strongly in $L^{3}(Q(1))$. By \eqref{R1}--\eqref{R2}, we have
\begin{align}\label{control2}
	v_{3}=B_{3}=0.
\end{align}
By Lemma \ref{A8}, the solution $(\bm{v},\bm{B})$ is regular at $(0,0)$. If we assume that the solution $(\bm{u},\bm{b})$ is singular at $(0,0)$, $(\bm{u}_{k},\bm{b}_{k})$ is also singular at $(0,0)$. It leads to a contradiction by Lemma \ref{A10}. 

The proof of Theorem \ref{thm1} is completed.\qed

\appendix
\section{}
\begin{lem}\cite[Lemma A.2]{Chae2021}\label{A1}
	Let $0<r \leqslant R<+\infty $ and $h: B^{\prime}(2 R) \times(-r, r) \rightarrow \R$ be harmonic. Then for all $0<\rho \leqslant \frac{r}{4}$ and $1 \leqslant \ell \leqslant q <+\infty$, we get
	\begin{align}
		\| h\|_{L^{q}\left(B^{\prime}(R) \times(-\rho, \rho)\right)}^{q} \leqslant C \rho r^{2- \frac{3q}{\ell}}\|h\|_{L^{\ell}\left(B^{\prime}(2 R) \times(-r, r)\right)}^{q},
	\end{align}
	where $C$ stands for a positive constant depending only on $q$ and $\ell$.
\end{lem}

We will present the estimates of $J$, $K$ and $H$, which is defined in \eqref{pressure}, in the following several lemmas.
\begin{lem}\label{A2}  Under the assumptions of Lemma \ref{thm3}, we have
	\begin{align}\label{J}
		J\leqslant C\sum_{i=0}^{n} \B_{i} \left(r_{i}^{-1} E_{i}(R)\right)+C\E^{\frac{3}{2}}.
	\end{align}
\end{lem}
\begin{proof}
	Let $\pi_{0,j}=\J\left(\left(\bm{u}\otimes\bm{u}-\bm{b}\otimes\bm{b}\right)\cdot\phi_{j}\right)$.
	\begin{align}\label{J'}
		J=&\sum_{k=0}^{n}\int_{-1}^{t} \int_{U_{0}(R)} \pi_{0}\cdot u_{3} \cdot \partial_{3}\mathrm{\Phi}_{n} \phi_{k} \psi\dx\ds      \notag\\
		=&\sum_{k=4}^{n}\sum_{j=0}^{n}\int_{-1}^{t} \int_{U_{0}(R)}  \pi_{0,j}\cdot u_{3} \cdot \partial_{3}\mathrm{\Phi}_{n} \phi_{k} \psi\dx  \text{d}s  \notag\\
		&+\sum_{k=0}^{3}\int_{-1}^{t} \int_{U_{0}(R)} \pi_{0}\cdot u_{3} \cdot \partial_{3}\mathrm{\Phi}_{n} \phi_{k} \psi\dx  \text{d}s  \notag\\
		=&\sum_{k=4}^{n}\sum_{j=0}^{k-4}\int_{-1}^{t} \int_{U_{0}(R)}  \pi_{0,j}\cdot u_{3} \cdot \partial_{3}\mathrm{\Phi}_{n} \phi_{k} \psi\dx  \text{d}s  \notag\\
		&+\sum_{k=4}^{n}\sum_{j=k-3}^{n}\int_{-1}^{t} \int_{U_{0}(R)}  \pi_{0,j}\cdot u_{3} \cdot \partial_{3}\mathrm{\Phi}_{n} \phi_{k} \psi\dx  \text{d}s  \notag\\
		&+\sum_{k=0}^{3}\int_{-1}^{t} \int_{U_{0}(R)} \pi_{0}\cdot u_{3} \cdot \partial_{3}\mathrm{\Phi}_{n} \phi_{k} \psi\dx  \text{d}s  \notag\\
		\eqdefa &J_{1}+J_{2}+J_{3}.
	\end{align}	
	By \eqref{Phi_{n}}, \eqref{ani}, Lemma \ref{energy} and Lemma \ref{A1} with $\rho=r_{k}, r=r_{j+2}$, we have
	\begin{align}\label{J_{1}}
		J_{1}
		=&\sum_{j=0}^{n-4}\sum_{k=j+4}^{n}\int_{-1}^{t} \int_{U_{0}(R)} \pi_{0,j}\cdot u_{3} \cdot \partial_{3}\mathrm{\Phi}_{n} \phi_{k} \psi\dx  \ds  \notag\\
		\lesssim&\sum_{j=0}^{n-4}\sum_{k=j+4}^{n} r_{k}^{-2}\int_{-r_{k}^2}^{0}\|\pi_{0,j}\|_{L^{\frac{q_{0}}{q_{0}-1}}(U_{k}(R))}\|u_{3}\|_{L^{q_{0}}(U_{k}(R))} \ds  \notag\\
		\lesssim&\sum_{j=0}^{n-4}\sum_{k=j+4}^{n} r_{k}^{-1-\frac{1}{q_{0}}}r_{j}^{-1-\frac{1}{2q_{0}}}\int_{-r_{k}^2}^{0}\|\left(\bm{u},\bm{b}\right)\|_{L^{\frac{4q_{0}}{2q_{0}-1}}(U_{j}(R))}^{2}\|u_{3}\|_{L^{q_{0}}(U_{j}(R))} \ds  \notag\\
		\lesssim&\sum_{j=0}^{n-4}\sum_{k=j+4}^{n} r_{k}^{1-\frac{2}{p_{0}^*}-\frac{5}{2q_{0}}}r_{j}^{-1-\frac{1}{2q_{0}}}\|\left(\bm{u},\bm{b}\right)\|_{L^{\frac{8q_{0}}{3}}\left(-r_{j}^2,0;L^{\frac{4q_{0}}{2q_{0}-1}}(U_{j}(R))\right)}^2\notag\\
		&\qquad\qquad\quad \times\|u_{3}\|_{L^{p_{0}^*}\left(-r_{j}^2,0;L^{q_{0}}(U_{j}(R))\right)}\notag\\
		\lesssim&\sum_{i=0}^{n} \B_{i} \left(r_{i}^{-1} E_{i}(R)\right).
	\end{align}
	By \eqref{Phi_{n}}, \eqref{ani} and Lemma \ref{energy}, we have
	\begin{align}\label{J_{2}}
		J_{2}=&\sum_{k=4}^{n}\int_{-1}^{t} \int_{U_{0}(R)} \J\left(\left(\bm{u}\otimes\bm{u}-\bm{b}\otimes\bm{b}\right)\cdot\chi_{k-3}\right)\cdot u_{3} \cdot \partial_{3}\mathrm{\Phi}_{n} \phi_{k} \psi\dx  \ds  \notag\\
		\lesssim&\sum_{k=4}^{n}r_{k}^{-2}\int_{-r_{k}^2}^{0}\|\J\left(\left(\bm{u}\otimes\bm{u}-\bm{b}\otimes\bm{b}\right)\cdot\chi_{k-3}\right)\|_{L^{\frac{q_{0}}{q_{0}-1}}(U_{k}(R))}\|u_{3}\|_{L^{q_{0}}(U_{k}(R))} \ds  \notag\\
		\lesssim& \sum_{k=0}^{n} r_{k}^{-2}\int_{-r_{k}^2}^{0}\|\left(\bm{u},\bm{b}\right)\|_{L^{\frac{2q_{0}}{q_{0}-1}}(U_{k}(R))}^{2} \|u_{3}\|_{L^{q_{0}}(U_{k}(R))} \ds  \notag\\
		\lesssim&\sum_{i=0}^{n} \B_{i} \left(r_{i}^{-1} E_{i}(R)\right),
	\end{align}
	which is analogous to \eqref{I_{1}}. By \eqref{Phi_{n}} and \eqref{ani}, we have
	\begin{align}\label{J_{3}}
		J_{3}\lesssim\int_{-1}^{0}\|\pi_{0}\|_{L^{\frac{3}{2}}(U_{0}(R))} \|u_{3}\|_{L^{3}(U_{0}(R))} \ds\lesssim\int_{-1}^{0}\|\left(\bm{u},\bm{b}\right)\|_{L^{3}(U_{0}(R))}^3 \ds\lesssim\E^{\frac{3}{2}}.
	\end{align}
	
	Summing up all the estimates of \eqref{J'}, \eqref{J_{1}}, \eqref{J_{2}} and \eqref{J_{3}}, we have \eqref{J}.\qed
\end{proof}

\begin{lem}\label{A3} Under the assumptions of Lemma \ref{thm3}, we have
	\begin{align}\label{K}
		K\leqslant \frac{C\E^{\frac{1}{2}}}{R-\rho} \sum_{i=0}^{n} r_{i}^{\frac{1}{2}}\left(r_{i}^{-1} E_{i}(R)\right)+\frac{C}{R-\rho}\E^{\frac{3}{2}}.
	\end{align}
\end{lem}
\begin{proof}
	\begin{align}\label{K'}
		K=&\int_{-1}^{t} \int_{U_{0}(R)} \pi_{0} \bm{u} \cdot \mathrm{\Phi}_{n} \eta \nabla\psi\dx  \ds +\int_{-1}^{t} \int_{U_{0}(R)} \pi_{0} \bm{u} \cdot \mathrm{\Phi}_{n} \nabla\eta \psi\dx  \ds\notag\\
		\lesssim&\sum_{k=0}^{n}\int_{-1}^{t} \int_{U_{0}(R)} \pi_{0} \bm{u} \cdot \mathrm{\Phi}_{n} \phi_{k} \nabla\psi\dx  \ds+\int_{-1}^{0}\int_{U_{0}(R)}|\pi_{0}| |\bm{u}|\dx\ds\notag\\
		\lesssim&\sum_{k=4}^{n}\sum_{j=0}^{n}\int_{-1}^{t} \int_{U_{0}(R)} \pi_{0,j} \ \bm{u} \cdot \mathrm{\Phi}_{n} \phi_{k} \nabla\psi\dx  \ds\notag\\
		&+\sum_{k=0}^{3}\int_{-1}^{t} \int_{U_{0}(R)} \pi_{0} \bm{u} \cdot \mathrm{\Phi}_{n} \phi_{k} \nabla\psi\dx  \ds+\int_{-1}^{0}\int_{U_{0}(R)}|\pi_{0}||\bm{u}|\dx\ds\notag\\
		\lesssim&\sum_{k=4}^{n}\sum_{j=0}^{k-4}\int_{-1}^{t} \int_{U_{0}(R)} \pi_{0,j} \ \bm{u} \cdot \mathrm{\Phi}_{n} \phi_{k} \nabla\psi\dx  \ds\notag\\
		&+\sum_{k=4}^{n}\sum_{j=k-3}^{n}\int_{-1}^{t} \int_{U_{0}(R)} \pi_{0,j} \ \bm{u} \cdot \mathrm{\Phi}_{n} \phi_{k} \nabla\psi\dx  \ds+\frac{1}{R-\rho}\int_{-1}^{0}\int_{U_{0}(R)}|\pi_{0}||\bm{u}|\dx\ds\notag\\
		\eqdefa &K_{1}+K_{2}+K_{3}.
	\end{align}
	Analogously with \eqref{J_{1}}, we have
	\begin{align}\label{K_{1}}
		K_{1}=&\sum_{j=0}^{n-4}\sum_{k=j+4}^{n}\int_{-1}^{t} \int_{U_{0}(R)} \pi_{0,j} \ \bm{u} \cdot \mathrm{\Phi}_{n} \phi_{k} \nabla\psi\dx  \ds\notag\\
		\lesssim&\frac{1}{R-\rho}\sum_{j=0}^{n-4}\sum_{k=j+4}^{n}r_{k}^{-1}\int_{-r_{k}^2}^{0}\|\pi_{0,j}\|_{L^{\frac{3}{2}}(U_{k}(R))}\|\bm{u}\|_{L^{3}(U_{k}(R))}\ds\notag\\
		\lesssim&\frac{1}{R-\rho}\sum_{j=0}^{n-4}\sum_{k=j+4}^{n}r_{k}^{-\frac{1}{3}}r_{j}^{-\frac{2}{3}}\int_{-r_{k}^2}^{0}\|\left(\bm{u},\bm{b}\right)\|_{L^{3}(U_{j}(R))}^{3}\ds\notag\\
		\lesssim&\frac{1}{R-\rho}\sum_{j=0}^{n-4}\sum_{k=j+4}^{n}r_{k}^{\frac{1}{6}}r_{j}^{-\frac{2}{3}}\|\left(\bm{u},\bm{b}\right)\|_{L^{4}\left(-r_{k}^2,0;L^{3}(U_{j}(R))\right)}^{3}\notag\\
		\lesssim&
		\frac{\E^{\frac{1}{2}}}{R-\rho} \sum_{i=0}^{n} r_{i}^{\frac{1}{2}}\left(r_{i}^{-1} E_{i}(R)\right).
	\end{align}
	Analogously with \eqref{J_{2}}, we have
	\begin{align}\label{K_{2}}
		K_{2}=&\sum_{k=4}^{n}\int_{-1}^{t} \int_{U_{0}(R)} \J\left(\left(\bm{u}\otimes\bm{u}-\bm{b}\otimes\bm{b}\right)\cdot\chi_{k-3}\right) \ \bm{u} \cdot \mathrm{\Phi}_{n} \phi_{k} \nabla\psi\dx  \ds\notag\\
		\lesssim&\frac{1}{R-\rho}\sum_{k=0}^{n}r_{k}^{-1}\int_{-r_{k}^2}^{0}\|\left(\bm{u},\bm{b}\right)\|_{L^{3}(U_{k}(R))}^3\ds\notag\\
		\lesssim&\frac{\E^{\frac{1}{2}}}{R-\rho} \sum_{i=0}^{n} r_{i}^{\frac{1}{2}}\left(r_{i}^{-1} E_{i}(R)\right).
	\end{align}
	Analogously with \eqref{J_{3}}, we have
	\begin{align}\label{K_{3}}
		K_{3}\lesssim\frac{1}{R-\rho}\int_{-1}^{0}\|\left(\bm{u},\bm{b}\right)\|_{L^{3}(U_{0}(R))}^3 \ds\lesssim\frac{1}{R-\rho}\E^{\frac{3}{2}}.
	\end{align}
	Summing up \eqref{K'}, \eqref{K_{1}}, \eqref{K_{2}} and \eqref{K_{3}}, we have \eqref{K}.\qed
\end{proof}

\begin{lem}\label{A4}  Under the assumptions of Lemma \ref{thm3}, we have
	\begin{equation}\label{H}
		H\leqslant \frac{C}{(R-\rho)^{3}} \E^{\frac{3}{2}}.
	\end{equation}
\end{lem}
\begin{proof}
	\begin{align}\label{3.25}
		H=&-\sum_{k=0}^{n}\int_{-1}^{t}\int_{U_{0}(R)} \nabla\pi_{h}\cdot \bm{u}\cdot  \left(\mathrm{\Phi}_{n} \phi_{k} \psi\right)\dx  \ds    \notag\\
		=&\sum_{k=0}^{3}\int_{-1}^{t}\int_{U_{0}(R)} \pi_{h}\cdot \bm{u}\cdot  \nabla\left(\mathrm{\Phi}_{n} \phi_{k} \psi\right)\dx  \ds  -\sum_{k=4}^{n}\int_{-1}^{t}\int_{U_{0}(R)} \nabla\pi_{h}\cdot \bm{u}\cdot  \left(\mathrm{\Phi}_{n} \phi_{k} \psi\right)\dx  \ds  \notag\\
		\eqdefa&H_{1}+H_{2}.
	\end{align}
	Applying \eqref{Phi_{n}}, H\"{o}lder's inequality and the fact that
	\begin{align}
		\|\pi_{h}\|_{L^{\frac{3}{2}}(\R^3)}\leqslant&\|\pi\|_{L^{\frac{3}{2}}(\R^3)}+\|\pi_{0}\|_{L^{\frac{3}{2}}(\R^3)}\lesssim\|\left(\bm{u},\bm{b}\right)\|_{L^{3}(\R^3)}^2,
	\end{align}
	we have
	\begin{align}\label{3.26}
		H_{1}\lesssim\frac{1}{R-\rho}\int_{-1}^{t}\int_{U_{0}(R)}|\pi_{h}||\bm{u}|\dx  \ds 
		\lesssim \frac{1}{R-\rho}\|\pi_{h}\|_{L^{\frac{3}{2}}\left(Q_{0}(1)\right)}\cdot\|\bm{u}\|_{L^{3}\left(Q_{0}(1)\right)}
		\lesssim\frac{1}{R-\rho} \E^{\frac{3}{2}}.
	\end{align}
	Moreover, 
	\begin{align}\label{5.6}
		H_{2}\lesssim&  \sum_{k=4}^{n}  r_{k}^{-1} \int_{Q_{k}(\frac{R+\rho}{2}) }|\nabla \pi_{h}| \cdot|\bm{u}|\dx  \ds  \notag\\
		\lesssim&  \sum_{k=4}^{n}  r_{k}^{-1}\left\|\nabla \pi_{h}\right\|_{L^{\frac{3}{2}}\left(Q_{k}(\frac{R+\rho}{2}) \right)}\|\bm{u}\|_{L^{3}\left(Q_{k}(R) \right)}\notag\\
		\lesssim& \sum_{k=4}^{n} r_{k}^{-\frac{1}{3}}\left\|\nabla \pi_{h}\right\|_{L^{\frac{3}{2}}\left(-r_{k}^{2}, 0 ; L^{\infty}\left(U_{k}(\frac{R+\rho}{2}) \right)\right)}\|\bm{u}\|_{L^{3}\left(Q_{k}(R) \right)}.
	\end{align}
	For any $x^{*} \in U_{k}(\frac{R+\rho}{2})$, we have $ B\left(x^* , \frac{R-\rho}{4}\right) \subset U_{1}(R)$
	due to $k \geqslant 4$ and $|R-\rho| \leqslant \frac{1}{2} .$
	Since $\pi_{h}$ is harmonic in $U_{1}(R),$ using the mean value property, we have
	\begin{align*}
		\left|\nabla \pi_{h}\right|\left(x^{*}\right)  \lesssim  \frac{1}{|R-\rho|^{4}} \int_{B\left(x^* , \frac{R-\rho}{4}\right)}|\pi_{h}|\dx 
		\lesssim \frac{1}{(R-\rho)^{3}}\left\|\pi_{h}\right\|_{L^ \frac{3}{2}\left(U_{1}(R) \right)}.
	\end{align*}
	Hence,
	\begin{align}\label{3.30}
		H_{2}\lesssim& \frac{1}{(R-\rho)^{3}} \sum_{k=4}^{n} r_{k}^{-\frac{1}{3}}\left\|\pi_{h}\right\|_{L^{\frac{3}{2}}\left(-r_{k}^{2}, 0 ; L^{\frac{3}{2}}\left(U_{1}(R)\right)\right)}\|\bm{u}\|_{L^{3}\left(Q_{k}(R)\right)} \notag\\
		\lesssim& \frac{1}{(R-\rho)^{3}} \sum_{k=4}^{n} r_{k}^{\frac{1}{6}}\ \|\left(\bm{u},\bm{b}\right)\|_{L^{4}\left(-r_{k}^{2}, 0 ; L^{3}\left(\R^3\right)\right)}^{3}\notag\\
		\lesssim& \frac{1}{(R-\rho)^{3}}\  \E^{\frac{3}{2}}.
	\end{align}
	Summing up \eqref{3.25}, \eqref{3.26} and \eqref{3.30} , we have \eqref{H}.\qed
\end{proof}

\begin{lem}\cite[Lemma A.3]{Wang2020}\label{A5}
	For any
	\begin{align*}
		1 \leqslant p_*<p=\frac{2q}{q-3}, \quad 3<q<+\infty,
	\end{align*}
	we have
	\begin{align*}
		\sum_{k=0}^{+\infty} r_{k}^{1-\frac{2}{p_*}-\frac{3}{q}}\left(\int_{-r_{k}^{2}}^{0}\left\|f(\cdot, s)\right\|_{L^q(B(2))}^{p_*} ~ds\right)^{\frac{1}{p_*}} \leqslant C\left\|f \right\|_{L^{p, 1}\left(-1,0;L^{q}(B(2))\right) }.
	\end{align*}
\end{lem}
For the sake of simplicity, we define
\begin{align}
	&
	D(\pi,z_{0},r)=r^{-2}\int_{Q(z_{0},r)}\left|\pi\right|^{\frac{3}{2}}\dx\dt,\ F(\bm{u},\bm{b},z_{0},r)=r^{-2}\int_{Q(z_{0},r)}|\bm{u}|^3+|\bm{b}|^3\dx\dt,\label{CD1}\\
	& D(\pi,r)\eqdefa D(\pi,(0,0),r), \ F(\bm{u},\bm{b},r)\eqdefa C(\bm{u},\bm{b},(0,0),r).\label{CD2}
\end{align}
\begin{lem} \label{A6}
	Let $\pi \in L^{\frac{3}{2}}(Q(1))$ solve $\Delta \pi=\nabla\cdot\nabla\cdot\left( \bm{u}\otimes\bm{u}-\bm{b}\otimes\bm{b}\right)$ in the sense of distributions. If there exists a constant $K_{0}$ such that for all $0< r \leqslant R\leqslant 1$,
	\begin{align}\label{sA.18}
		F(\bm{u},\bm{b},r)\leqslant K_{0},
	\end{align}
	then for some $\alpha>0$ and all $0<r \leqslant R$,
	\begin{align}\label{sA.19}
		D(\pi,r) \leqslant C \left(\frac{r}{R}\right)^{\alpha}\cdot D(\pi,R)+C K_{0}.
	\end{align}
\end{lem}
\begin{proof}
	We claim that for $0<2r< \rho\leqslant R$,
	\begin{align}\label{sA.20}
		D(\pi,r)\leq C_{2}\left(\frac{r}{\rho}\right) D(\pi,\rho)+C_{2}\left(\frac{\rho}{r}\right)^2 F(\bm{u},\bm{b},\rho).
	\end{align}
	Actually, we write $\pi=\pi_{0}+\pi_{h}$,
	where $\pi_{0}=\J\left(\left( \bm{u}\otimes\bm{u}-\bm{b}\otimes\bm{b}\right) \cdot \chi_{B(\rho)}\right)$ and
	\begin{align}\label{sA.21}
		D(\pi_{0},r)\lesssim r^{-2}\| \bm{u}\otimes\bm{u}-\bm{b}\otimes\bm{b}\|_{L^{\frac{3}{2}}(Q(\rho))}^\frac{3}{2}\lesssim \left(\frac{\rho}{r}\right)^2 F(\bm{u},\bm{b},\rho).
	\end{align}
	$\pi_{h}$ is harmonic in $B(\rho)$ and the mean value property of $\pi_{h}$ implies that
	\begin{align}\label{sA.22}
		D(\pi_{h},r)\lesssim r \int_{-r^2}^{0}\|\pi_{h}\|_{L^{\infty}(B(r))}^{\frac{3}{2}}\dt\lesssim \left(\frac{r}{\rho}\right) D(\pi_{h},\rho)\lesssim\left(\frac{r}{\rho}\right)\left( D(\pi,\rho)+D(\pi_{0},\rho)\right).
	\end{align}
	Summing up \eqref{sA.21} and \eqref{sA.22}, we have \eqref{sA.20}.
	
	Let $\theta \in(0,\frac{1}{2})$. By \eqref{sA.18} and \eqref{sA.20}, we have that for $0<r\leqslant R$,
	\begin{align}
		D(\pi,\theta r)\leqslant C_{2}\cdot \theta\  D(\pi,r)+C_{2}K_{0}\ \theta^{-2}.
	\end{align}
	We choose $\theta$ such that $C_{2}\cdot \theta = \frac{1}{2}$.  For $\theta R < r\leq R$, we have
	\begin{align}
		D(\pi,r)\leqslant \theta^{-2} R^{-2}\int_{Q(R)}|\pi|^{\frac{3}{2}}\dx\dt.
	\end{align}
	This together with a standard iteration yields \eqref{sA.19} with $\alpha=-\frac{\ln 2}{\ln \theta}>0$.
	\qed
\end{proof}
\begin{lem}\cite[Theorem 1.1]{Vyalov2008}\label{A9}
	There exists an absolute constant $\varepsilon_{0}>0$ with the following property. If $(\bm{u},\bm{b},\pi)$ is a suitable weak solution to the MHD equations in $Q(1)$ satisfying that for some $0<R\leqslant 1$, $Q(z_{0},R)\subseteq Q(1)$ and
	\begin{align}\label{claim1}
		R^{-2} \int_{Q(z_{0},R)}\left(|\bm{u}|^{3}+|\bm{b}|^{3}+|\pi|^{\frac{3}{2}}\right) \dx\dt<\varepsilon_{0},
	\end{align}
	then the solution is regular at $z_{0}=(x_{0},t_{0})$.
\end{lem}
\begin{lem}[stability of singularities]\label{A10}
	Let $\left(\bm{u}_{k},\bm{b}_{k}, \pi_{k}\right)$ be a sequence of suitable weak solutions to the MHD equations \eqref{MHD} in $Q(1)$ such that $(\bm{u}_{k},\bm{b}_{k}) \rightarrow (\bm{v},\bm{B})$ in $L^{3}(Q(1))$, $\pi_{k} \rightharpoonup \Pi$ in $L^{\frac{3}{2}}(Q(1))$. Assume $(\bm{u}_{k},\bm{b}_{k})$ is singular at $z_{k}=(x_{k},t_{k})$, $z_{k}\rightarrow(0,0)$ as $k\rightarrow\infty$. Then $(\bm{v},\bm{B})$ is singular at $(0,0)$.
\end{lem}
\begin{proof}
	The proof is similar with \cite[Lemma 2.1]{Rusin2011}. If $(\bm{v},\bm{B})$ is regular at $(0,0)$, then there exists  $\rho_{0}>0$ and for all $0<r<\rho_{0}$,
	\begin{align}\label{A.280}
		r^{-2}\int_{Q(r)}|\bm{v}|^3+|\bm{B}|^3\dx\dt\leqslant C \ r^3.
	\end{align}
	Since $(\bm{u}_{k},\bm{b}_{k})$ is singular at  $z_{k}$, by \eqref{CD1} and Lemma \ref{A9}, we have that for all $0<r< \frac{1}{2}$,
	\begin{align}
		F(\bm{u}_{k},\bm{b}_{k},z_{k},r)+D(\pi_{k},z_{k},r)\geqslant \varepsilon_{0}.
	\end{align}
	For sufficient large $N=N(r)$ and all $k\geqslant N$ , we have $Q(z_{k},r)\in Q(2r)$ and
	\begin{align}
		F(\bm{u}_{k},\bm{b}_{k},2r)+D(\pi_{k},2r)\geqslant \frac{\varepsilon_{0}}{4} .
	\end{align}
	Denote $\tilde{F}(r)=\underset{k\rightarrow\infty}{\limsup}\ F(\bm{u}_{k},\bm{b}_{k},r)$ and $\tilde{D}(r)=\underset{k\rightarrow\infty}{\limsup}\ D(\pi_{k},r)$. For all $0<r< 1$, we have
	\begin{align}
		\tilde{F}(r)+\tilde{D}(r)\geqslant\frac{\varepsilon_{0}}{4}.
	\end{align}
	By \eqref{A.280}, we have that for all $0<r< \rho<\rho_{0}$,
	\begin{align}
		\tilde{F}(r)=F(\bm{v},\bm{B},r)\leqslant Cr^3\leqslant C \rho^3.
	\end{align}
	Applying an analogous argument in Lemma \ref{A6}, we have that for all $r$ with $0<r< \rho$,
	\begin{align}
		\tilde{D}(r) \leqslant C \left(\frac{r}{\rho}\right)^{\alpha}\cdot \tilde{D}(\rho)+C \rho^3.
	\end{align}
	Accordingly, we have for all $0<r< \rho<\rho_{0}$,
	\begin{align}
		C_{3} \left(\frac{r}{\rho}\right)^{\alpha}\cdot \tilde{D}(\rho)+C_{3} \rho^3 \geqslant \frac{\varepsilon_{0}}{4},
	\end{align}
	where the constant $C_{3}$ is independent of $r$ and $\rho$. It leads to a contradiction if we let $r\rightarrow0$ and then $\rho\rightarrow0$. The proof is completed.	\qed
\end{proof}
\begin{lem}\cite[Theorem 1.1]{Wang2013}\label{A7}
	There is an absolute number $\varepsilon_{1}>0$ with the following property. If $(\bm{u},\bm{b},\pi)$  is a suitable weak solution of \eqref{MHD} in $Q(1)$ satisfying that for some $0<R_{0} < 1$ and all $0<r<R_{0}$,
	\begin{align}
		r^{-2} \int_{Q(r)}|\bm{u}|^{3}\dx\dt < \varepsilon_{1},
	\end{align}
	then the solution $(\bm{u},\bm{b})$ is regular at $(0,0)$.
\end{lem}

\section*{Acknowledgments}
	Hui Chen was supported by Natural Science Foundation of Zhejiang Province(LQ19A010002). Chenyin Qian was supported by Natural Science Foundation of Zhejiang Province(LY20A010017). Ting Zhang was in part supported by National Natural Science Foundation of China (11771389, 11931010, 11621101). 
\bibliography{2021MHD}
\bibliographystyle{abbrv}
\end{document}